\documentclass{article}
\usepackage{amssymb}
\usepackage{amsmath}
\usepackage{amsthm}
\usepackage{verbatim}
\usepackage{color}
\usepackage{url}
 \usepackage[all]{xy}
\usepackage{fancyhdr}

\newtheorem{thm}{Theorem}

\newtheorem{lem}{Lemma}
\newtheorem{cor}{Corollary}

\newcommand{\sabs}[1]{\left|#1\right|}
\newcommand{\sparen}[1]{\left(#1\right)}
\newcommand{\norm}[1]{\sabs{\sabs{#1}}}

\numberwithin{equation}{section}
 
\title{Low Regularity Ray Tracing for Wave Equations with Gaussian beams}
\author{Alden Waters\thanks{amswaters@rug.nl, Bernoulli Institute, University of Groningen
Faculty of Science and Engineering (FSE)}} 
\begin{document}
\maketitle

\begin{abstract}
We prove observability estimates for oscillatory Cauchy data modulo a small kernel for $n$-dimensional wave equations with space and time dependent $C^2$ and $C^{1,1}$ coefficients using Gaussian beams. We assume the domains and observability regions are in $\mathbb{R}^n$, and the GCC applies. This work generalizes previous observability estimates to higher dimensions and time dependent coefficients. The construction for the Gaussian beamlets solving $C^{1,1}$ wave equations represents an improvement and simplification over Waters (2011). 
\end{abstract}

\textbf{keywords:} control theory, inverse problems, wave equations, observability, low regularity coefficients, Gaussian beams\\

\textbf{MSC classes:} 35R01, 35R30, 35L20, 58J45, 35A22

\section{Introduction} 

Significant interest in inverse problems and imaging centers around low-regularity wave equations. Often times in nature equations governing physical situations can be irregular and discontinuities are difficult to model. Ray tracing techniques used in geophysics often require $C^{3}$ coefficients to model the physical phenomenon. These techniques have been popular with the inverse problems community \cite{AS, Isakov, Isakov2,Rakesh, Rakesh2, SU1, SU2}.  The goal of this article is to introduce a construction which generalizes the Gaussian beam Ansatz and is suitable for proving multi-dimensional observability estimates for low regularity equations. The relationship of these observability estimates to hyperbolic inverse boundary value problems is clear. One can 'observe' or recover a source from following the geometric optics rays from the observation back to the source. This technique is known as ray-tracing. Our results indicate that in some cases, for practical applications imaging is still possible. The main theorems extend the observability estimates of \cite{fz} both in dimension and generality of the geometry. We use the observability criterion of Bardos-Lebeau-Rauch, \cite{blr} and a non-gliding hypothesis on the rays.

Specifically, we prove in this paper observability estimates for $n$-dimensional classical wave equations with $C^{1,1}$ and $C^2$ time dependent coefficients. Whenever the coefficients are in these low regularity classes we establish observability estimates even though the typical assumption is $C^3$ coefficients. We will give a precise idea of why this is true when specifying the error estimates for the construction of the approximate solution, or Ansatz solution. This appears to be the first case in which space and time dependent coefficients are also examined in the context of observability estimates. 

These estimates are proved for one-dimensional wave equations by means of a uniquely one-dimensional technique \cite{fz} called sidewise energy estimates where the use of space and time are interchanged. We use same idea that the use of time and arc-length are interchanged which was developed in \cite{waterss}. This development seems like a natural generalization of sidewise energy estimates. 

In higher space dimensions the problem is more complex and other techniques are required. Bardos-Lebeau-Rauch \cite{blr} proved it is necessary and sufficient for observability estimates to hold for wave equation solutions if the observability region satisfies the geometric control condition.

For this article we assume $\Omega$ is a domain in $\mathbb{R}^n$, and $\Omega_0\subset\Omega$ is the observation region. In general the work of Lebeau showed that control from $\Omega_0\subset\Omega$ (the dual statement to the existence of the observability constants for the wave equation)  holds for the wave equation under the Geometric Control Condition (GCC):

$\bullet$ There exists $L=L(\Omega,\Omega_0)>0$ such that every Hamiltonian ray path of length $L$ on $\Omega$ intersects $\Omega_0$. \\

In one dimension, rays can only travel in a sidewise manner, so this condition is automatic to verify. However for the higher dimensional case the development of more complex Ansatzes to trace the ray path in space-time involves $C^2$ \cite{Ralston} or at minimum $C^{1,1}$ \cite{watersp} coefficients. The main novelty of this paper is an Ansatz which does not require the usual $C^3$ coefficients necessary for the construction of the localized tail. 
 
 For the system in question, observability estimates are known to be verified in 1d when the coefficients are in bounded variation class, see \cite{fzbv}. Observability estimates with a loss of derivatives continue to hold for coefficients in log-Zygmund spaces \cite{fz}, but fail for $C^\alpha$ \cite{cz} with $\alpha<1$. In higher dimensions, Burq extended the observability estimates to $C^2$ time independent  coefficients and $C^3$ domains \cite{burq}. For this work, we concentrate on results in the positive direction, and use an equivalent hypothesis on the geometry of the rays as in \cite{burq}. In contrast to previous works, we assume that Cauchy data has an oscillatory phase part. This assumption could be removed by using the frame in \cite{watersp}, and this will be reported on in future work. 
 
\section{Statement of the Main Theorems}

We recall the definition of $C^{1,1}$ coefficients. We say the collection $\{g^{jk}(x,t)\}_{j,k=1}^n$ with $(x,t)\in \Omega\times [0,T]$ is  $C^{1,1}$ if the coefficients satisfy a Lipschitz condition in $x$ and $t$ 
\[
|g^{jk}(x,t)-g^{jk}(x',t')|\leq M(|t-t'|+|x-x'|)
\]
and their first derivatives in $x$ satisfy a Lipschitz condition 
\[
|\nabla_xg^{jk}(x,t)-\nabla_xg^{jk}(x',t)|\leq M_0|x-x'|
\]
for some positive constants $M, M_0$ independent of $x$ and $t$. In general we assume, the coefficient $g^{jk}(x,t)$ is variable over the set $\Omega\times [0,T]$, and equal to $\delta^{jk}$ outside a much larger set. 
 
Now we introduce our operator. Let $H(x,t,\partial_x,\partial_t)$ be a second order hyperbolic operator of the form:
\begin{align}\label{op}
&H(x,t,\partial_x,\partial_t)=
-\frac{\partial^2}{\partial t^2}+\sum\limits_{j,k=1}^ng^{jk}(x,t)\frac{\partial^2}{\partial x_j\partial x_k}.
\end{align}
We assume that there exists a constant $\alpha_0$ such that 
\begin{align}
\sum\limits_{j,k=1}^n g^{jk}(x,t)\xi_i\xi_j\geq \alpha_0|\xi|^2
\end{align}
with $g^{jk}(x,t)=\delta^{jk}$ on $\partial\Omega\times [0,T]$. Sometimes we will assume the coefficients are $g^{jk}(x,t)\in C^2(\Omega\times [0,T])$, or $C^{1,1}(\Omega\times [0,T])$, and we will specify the regularity as needed. 

For $x\in\Omega$ and $\omega\in T_x\Omega$ we denote initial conditions of the system of ODEs \eqref{bich}. We let 
\begin{align}
\mathcal{S}\Omega=\{(x,\omega)\in T\Omega;\quad |\omega|=1\}
\end{align}
denote the sphere bundle of $\Omega$. We consider pairs $(x_0,\omega_0)$ for which the bicharacteristic flow \eqref{bich} (instead of the geodesic flow) satisfies the GCC. However, it can be shown that using the coefficients $g^{jk}$ to endow $\Omega$ with a metric topology so that $(g,\Omega)$ is a manifold, that bicharateristic flow and the geodesic flow on the manifold are equivalent, c.f. the Appendix \cite{waterss} for this computation. 

We define a phase function $\psi_0(x)$ as:
\begin{align}
\psi_0(x)=(x-x_0)\cdot\omega_0+i|x-x_0|^2
\end{align}
and consider $\lambda\in\mathbb{R}^+$ a large asymptotic parameter. The vectors $(x_0,\omega_0)\in \mathbb{R}^n$ are in the admissible set defined above. 

We consider the Cauchy problem with Dirichlet boundary conditions:
\begin{align}\label{cauchy}
&Hu(x,t)=0 \quad \textrm{in} \quad \Omega\times\mathbb{R}_t^+\quad\\& \nonumber
u_0=u(x,0)=\lambda^{n/4}f_1(x)\exp(i\lambda\psi_0) \qquad u_1=\frac{\partial u}{\partial t}(x,0)=\lambda^{n/4}f_2(x)\exp(i\lambda\psi_0)\quad \textrm{in}\quad \Omega \\& \nonumber
u(x,t)=0 \quad \textrm{on}\quad \partial\Omega\times\mathbb{R}_t^+
\end{align}
where each function $f_k(x)\in C^3(\Omega)$ $k=1,2$ ($H^3$ regularity is also possible with slightly more work). It is well known under these conditions that this problem is well-posed for some finite time $T$, with $u\in C([0,T];\dot{H}_0^1(\Omega))\cap C^1([0,T];L^2(\Omega))$, c.f. Theorem \ref{energy} below.  

Let $\Omega_0\subset\Omega$ be such that $\Omega_0$ satisfies the geometric control condition. Let $\Omega\subset\Omega'$ with $\Omega'$ a larger domain. We let the constant $C_{12}$ be such that $C_{12}$ depends on $T, \mathrm{diam}_g(\Omega), \alpha_0$ and $||g^{jk}(x,t)||_{C^2(\Omega\times[0,T])}$ only. We also let $C_{22}$ be such that it depends on $||f_1||_{C^3(\Omega)}+||f_2||_{C^2(\Omega)},\mathrm{diam}_{g'}(\Omega')$, $\alpha_0$, $||g^{jk}(x,t)||_{C^2(\Omega\times[0,T])}$, and $T$. 

We have the following theorem, provided the set $\Omega_0\subset\Omega$ satisfies the GCC.

\begin{thm}\label{low1Rn}
Assume that $g^{jk}(x,t)$ are $C^2(\Omega\times[0,T])$, then for every $\epsilon>0$ there exists nonzero constants $C_{12}$ and $C_{22}$ as above such that for all $\lambda$ sufficiently large:
\begin{align}
||u_0||_{L^2(\Omega)}+||u_1||_{H^{-1}(\Omega)}\leq C_{12}||u||_{L^2(\Omega_0\times [0,T])}+\frac{C_{22}}{\lambda^{1/3}}+C_{22}C_2(\epsilon)
\end{align}
where $C_2(\epsilon)$ is a positive constant depending on $\epsilon$ such that $C_2(\epsilon)\rightarrow 0$ as $\epsilon\rightarrow 0$.
\end{thm}

\begin{cor}
As $\epsilon\rightarrow 0$ and $\lambda\rightarrow\infty$, this gives a perfect observability estimate:
\begin{align}
||u_0||_{L^2(\Omega)}+||u_1||_{H^{-1}(\Omega)}\leq C_{12}||u||_{L^2(\Omega_0\times [0,T])}
\end{align}
\end{cor}

We let the constant $C_{11}$ be such that $C_{11}$ depends on $T, \mathrm{diam}_g(\Omega), \alpha_0$, and $||g^{jk}(x,t)||_{C^{1,1}(\Omega\times[0,T])}$ only. We also let $C_{21}$ be such that it depends on $||f_1||_{C^3(\Omega)}+||f_2||_{C^2(\Omega)},\mathrm{diam}_{g'}(\Omega')$, $\alpha_0$, $||g^{jk}(x,t)||_{C^{1,1}(\Omega\times[0,T])}$ and $T$.

\begin{thm}\label{low1Rn2}
Assume that $g^{jk}(x,t)$ are $C^{1,1}(\Omega\times [0,T])$. For every $\epsilon>0$, there exists constants $C_{11}$ and $C_{21}$ such that for all $\lambda$ sufficiently large
\begin{align}
||u_0||_{L^2(\Omega)}+||u_1||_{H^{-1}(\Omega)}\leq C_{11}||u||_{L^2(\Omega_0\times[0,T])}+\frac{C_{21}}{\lambda^{1/3}}+C_1(M_0)
\end{align}
where $C_1(M_0)$ is a positive constant depending on $M_0$ such that $C_1(M_0)\rightarrow 0$ as $M_0\rightarrow 0$.
\end{thm}

\begin{cor}
As $M_0\rightarrow 0$, and $\lambda\rightarrow\infty$ this gives a perfect observability estimate:
\begin{align}
||u_0||_{L^2(\Omega)}+||u_1||_{H^{-1}(\Omega)}\leq C_{11}||u||_{L^2(\Omega_0\times [0,T])}
\end{align}
\end{cor}

Even though the second theorem requires less regularity, we show both Theorems for completeness because the computations will indicate that the higher the regularity the sharper the estimates. Moreover, if the modulus of continuity for the derivative of the metric is too large in the case of $C^{1,1}$ coefficients, estimates are not possible. In the second corollary, the equation is close to $\partial^2_tu-\Delta u=0$, for which the beamlets have a phase function solving the equation exactly. 

\section{Ansatz Construction}

 If the reader is expecting a typical Ansatz to the wave equation, that will not be found here as $C^3$ coefficients are necessary for easy control over the error terms, thus increasing the complexity of the construction. We modify the existing construction for Gaussian beams to low regularity coefficients. Gaussian beams are asymptotically valid high frequency solutions to hyperbolic differential equations which are concentrated on a single physical curve in the domain. One can extend them to dispersive equations such as the Schrodinger equation. Superpositions of Gaussian beams also provide a powerful tool to generate more general high frequency solutions which are not necessarily concentrated on a single curve. We present a systematic construction of Gaussian beam super-positions for the wave equation, with the caveat that we have extremely low regularity coefficients. Geometric optics breaks down at caustics where the rays concentrate and the predicted amplitude is unbounded. The consideration of such difficulties which are caused by caustics starting with Keller \cite{keller} and then Maslov and Fedoriuk \cite{mf} which leads to the development of Fourier integral operators given by Hormander \cite{hormander}. As such, we follow the parametrix construction in \cite{eskin} Section 64 losely, which is based on this Fourier integral operator construction. The major differences are that we have to be more careful about the treatment of our error terms, and we continue the construction to include the Gaussian beam tails. 

We let $H_0$ be the principal symbol of the operator \eqref{op}. One can factor the symbol as 
\begin{align}
H_0(x,t,\xi,\tau)=(\tau-\lambda_1(x,t,\xi))(\tau-\lambda_2(x,t,\xi))
\end{align}
where the roots are 
\begin{align}
\lambda_1=\sqrt{\sum\limits_{j,k=1}^ng^{jk}(x,t)\xi_j\xi_k}\quad \lambda_2=-\sqrt{\sum\limits_{j,k=1}^ng^{jk}(x,t)\xi_j\xi_k}.
\end{align}

We now proceed to construct a parametrix to \eqref{cauchy}. 
\begin{thm}\label{ansatz2}
Assume $g^{jk}(x,t)$ are $C^2(\Omega\times [0,T])$, then for every $\epsilon>0$ and $\lambda=\lambda(\epsilon)$ sufficiently large, there is an approximate solution $U_{\lambda}(x,t)$ to the problem \eqref{cauchy} such that 
\begin{align}
||u-U_{\lambda}(x,t)||_{L^2(\Omega\times [0,T])}\leq C_{22}C_2(\epsilon)
\end{align}
where $C_{22}$ depends on $||f_1||_{C^3(\Omega)}+||f_2||_{C^2(\Omega)},\,\,\mathrm{diam}_{g'}(\Omega')$, $T$, $\alpha_0$ and $||g^{jk}(x,t)||_{C^2(\Omega\times[0,T])}$. In particular, we have that $C_2(\epsilon)\rightarrow 0$ as $\epsilon\rightarrow 0$
\end{thm}
and also

\begin{thm}\label{ansatz1}
Assume $g^{jk}(x,t)$ are $C^{1,1}(\Omega\times[0,T])$. For every $\epsilon>0$ and $\lambda(\epsilon)$ sufficiently large, there is an approximate solution $U_{\lambda}(x,t)$ to the problem \eqref{cauchy} such that 
\begin{align}
||u-U_{\lambda}(x,t)||_{L^2(\Omega\times [0,T])}\leq  C_1(M_0)+\frac{C_{21}}{\sqrt\lambda}
\end{align}
where $C_{21}$ depends on $||f_1||_{C^3(\Omega)}+||f_2||_{C^2(\Omega)},\,\mathrm{diam}_{g'}(\Omega')$, $T$, $\alpha_0$ and $||g^{jk}(x,t)||_{C^{1,1}(\Omega\times[0,T])}$. Recall that $M_0$ is the modulus of continuity for $\nabla_xg^{jk}(x,t)$. In particular, we have that $C_1(M_0)\rightarrow 0$ as $M_0\rightarrow 0$.
\end{thm}
We call the factor $M_0$ the relative error in the problem. It is not possible to obtain the second set of estimates without the relative error. 

We concentrate on the $C^2$ coefficient case first. 
We look for a solution to \eqref{cauchy} of the form
\begin{align}\label{ansatz}
U_{\lambda}(x,t)=\sum\limits_{j=1}^2\sum\limits_{k=0}^2\lambda^{n/4}a_{jk}(x,t,\eta)\exp(i\psi_j(x,t,\eta))
\end{align}
where $a_{jk}(x,t,\eta)\in C^3(\Omega\times [0,T])$, and $\eta=\lambda\omega_0$, with $\lambda\in\mathbb{R}^+$ and $\omega_0\in \mathbb{R}^n$ such that $|\omega_0|=1$.  We let $v_j(x,t,\eta)\in C^3(\Omega\times [0,T])$ be the correction terms in the equation, that is 
\begin{align}
u(x,t)=U_{\lambda}(x,t)+\lambda^{n/4}\sum\limits_{j=1}^2v_{j}(x,t,\eta)\exp(i\psi_j(x,t,\eta))
\end{align}

Substituting \eqref{ansatz} into \eqref{cauchy} and grouping the powers of $\eta$ we get the following set of equations
\begin{align}\label{elk}
&H_0(x,t,\psi_{jx},\psi_{jt})=0 \nonumber\quad j=1,2 \\& 
\frac{\partial H_0(x,t,\psi_{jx},\psi_{jt})}{\partial\xi_0}\frac{\partial a_{j0}}{\partial t}+\sum\limits_{k=1}^n\frac{\partial H_0(x,t,\psi_{jx},\psi_{jt})}{\partial\xi_k}\frac{\partial a_{j0}(x,t,\eta)}{\partial x_k}+\left (H_0(x,t,\frac{\partial}{\partial x},\frac{\partial}{\partial t})\psi_j\right)a_{j0}\\& \label{trans}
L_ja_{jk}=-iH(x,t,\partial_x,\partial_t)a_{j,k-1} \quad k\geq 1, j=1,2 \nonumber
\end{align} 
where $L_j$ is the operator on the left and side of the second equation in \eqref{elk}. We need to be able to solve these equations to some order along a centralised curve. To this end, it helps if we use Cauchy data for \eqref{cauchy} with Gaussian packets rather than non-oscillatory $L^2$ functions.
 
In order to find the Cauchy conditions we look at the following initial conditions on $a_{jk}$ $j=1,2$. We obtain
\begin{align}
&\sum\limits_{j=1}^2\sum\limits_{k=0}^2a_{jk}(x,0,\eta)=f_1(x) \\& \nonumber
i\psi_{1t}(x,0,\eta)\sum\limits_{k=0}^2a_{1k}(x,0,\eta)+i\psi_{2t}(x,0,\eta)\sum\limits_{k=0}^2a_{2k}(x,0,\eta)+\\& \sum\limits_{j=1}^2\sum\limits_{k=0}^2\frac{\partial a_{jk}(x,0,\eta)}{\partial t}=f_2(x)
\end{align}

Because $\psi_j(x,0,\eta)=\lambda_j(x,0,\eta)$ for $j=1,2$ we can find $a_{10}(x,0,\eta)$ and $a_{20}(x,0,\eta)$ as the unique solution of the algebraic $2\times 2$ system
\begin{align}
& a_{10}(x,0,\eta)+a_{20}(x,0,\eta)=f_1(x) \\&
i\lambda_1(x,0,\eta)a_{10}+i\lambda_2(x,0,\eta)a_{20}=f_2(x)
\end{align}
as the determinant of this system $\lambda_2-\lambda_1\neq 0$, for all $x,t,\eta\neq 0$. We determine $a_{jk}(x,0,\eta)$ from the equations
\begin{align}
&a_{1k}(x,0,\eta)+a_{2k}(x,0,\eta)=0\\& 
i\lambda_1(x,0,\eta)a_{1k}+i\lambda_2(x,0,\eta)a_{2k}=-\frac{\partial a_{1,k-1}(x,0,\eta)}{\partial t}-\frac{\partial a_{2,k-1}(x,0,\eta)}{\partial t} \quad 1\leq k\leq 2 \nonumber
\end{align}
We see that 
\begin{align}
&HU_{\lambda}(x,t)=\sum\limits_{j=1}^2\lambda^{n/4}t_{j}(x,t,\eta)\exp(i\psi_j)
\end{align}
if we set $u-U_{\lambda}(x,t)=w(x,t,\eta)$ then it follows 
\begin{align}
&Hw(x,t,\eta)=-\sum\limits_{j=1}^2\lambda^{n/4}t_{j}(x,t,\eta)\exp(i\psi_j)\\& \nonumber
w(x,0,\eta)=0 \quad \frac{\partial w(x,0,\eta)}{\partial t}=-\lambda^{n/4}t_{j}(x,0,\eta)\exp(i\psi_j).
\end{align}
In order to bound the errors, we recall the following somewhat classical energy estimate. Let $u$ be a solution to the system with $f(x,t)\in L^2(\Omega\times [0,T])$, $u_0\in H^1(\Omega)$ and $u_1\in L^2(\Omega)$.
\begin{align}
&Hu=f(x,t) \quad \textrm{in} \quad \Omega\times [0,T] \\& \nonumber
u(x,0)=u_0 \quad \partial_tu(x,0)=u_1 \quad \mathrm{in} \quad \Omega \\& \nonumber
u(x,t)=0 \quad \mathrm{on} \quad \partial\Omega\times [0,T]
\end{align}
The following holds for the above Cauchy problem 
\begin{thm}\label{energy}
There exists $C$ depending on $\alpha_0$ and $||g^{jk}||_{C^{1,1}(\Omega\times [0,T])}$ and $\tilde{A}_1$ depending on $||g^{jk}||_{C^{1,1}(\Omega\times [0,T])}$ such that
\begin{align}
&||u||_{C([0,T];\dot{H}^1(\Omega))\cap C^1([0,T];L^2(\Omega))}\leq \\& \nonumber C\left(||u_0||_{H^1(\Omega)}+||u_1||_{L^2(\Omega)}+||f(x,t)||_{L^2(\Omega \times [0,T])}\right)\exp(\tilde{A}_1T). 
\end{align}
\end{thm} 
This particular variant is an exercise in integration by parts but one could c.f. Liones and Magenes \cite{lions}.  
In order to obtain the desired result, one needs to gain precision over $t_{j}$ for $j=1,2$ in order to use the energy estimate above. These estimates are where we use the regularity hypothesis to build our Gaussian beam. We chose the phase functions $\psi_j$ to be real valued along the curve $x=x(t)$, and we fix the initial value of the phase function as
\begin{align}\label{icephase}
\psi(0,x,\eta)=\left(i|x-x_0|^2/2+(x-x_0)\cdot \omega_0)\right|\eta|.
\end{align}
We assume $\psi$ is positive homogeneous of degree one in $|\eta|$, so we can write 
\begin{align}
\psi(x,0,\eta)=|\eta|\tilde{\psi}(0,x,\omega_0)=\lambda\tilde{\psi}(x,0,\omega_0)=\lambda\psi_0
\end{align} 
We now work with $\tilde{\psi}$ and drop the tilde everywhere where it is understood. 

We show we can write the phase function in a Taylor series.
\begin{lem}[from \cite{waterss}]
Given initial conditions \eqref{icephase}, for $\beta=1,2$ we have 
\begin{align*}
\psi^1_{\beta}(x,t)=(x-x(t))\cdot\omega(t) \qquad \psi^2_{\beta}(t,x)=M_{lj}(t)(x-x(t))^{l}(x-x(t))^{j} 
\end{align*}
where $M(t)$ is a matrix such that $\Im M(t)$ is positive definite, and $\psi_{\beta}(t,x)=\psi^1_{\beta}(x,t)+\psi^2_{\beta}(x,t)$ solves the eikonal (first equation in \eqref{elk}) to order $o(|x-x(t)|^2)$. 
\end{lem}

The first equation in \eqref{elk} can be recast in the following form
\begin{align}\label{hflow}
\psi_{1t}=h(x,t,\psi_{1x}).
\end{align}
Working backwards (and suppressing the 1 where it is understood, same is true for 2 but with $\lambda_2$ as the root of \eqref{elk}), if we differentiate the equation (\ref{hflow}) with respect to $x$ we obtain the following relations
\begin{align}\label{systempsi}
&\psi_{tx_j}-h_{p_i}(x,t,\psi_x)\psi_{x_ix_j}=h_{x_j}(x,t,\psi_x)\\ \nonumber
&\psi_{tt}-h_{p_i}(x,t,\psi_x)\psi_{x_it}=0\\ \nonumber
&\psi_{tx_jx_k}-h_{p_i}(x,t,\psi_x)\psi_{x_ix_jx_k}=\\ \nonumber
&h_{x_jx_k}(x,t,\psi_x)+h_{x_jp_i}(x,t,\psi_x)\psi_{x_ix_k}+h_{x_lp_i}(x,t,\psi_x)\psi_{x_ix_j}+h_{p_ip_m}(x,t,\psi_x)\psi_{x_ix_j}\psi_{x_mx_k}.
\end{align}
To simplify this set of relations we consider the matrices $A$, $B$ and $C$ which are defined with entries as follows:
\begin{align}\label{matrices}
& A_{j}^i=\{h_{x_ix_j}(x(t),t,\omega(t))\} \\ \nonumber
& B_j^i=\{h_{x_ip_j}(x(t),t,\omega(t))\}  \\ \nonumber
& C_j^i=\{h_{p_ip_j}(x(t),t,\omega(t))\}.  \nonumber
\end{align}
If we set $\nabla_x\psi(t,x(t))=\omega(t)$ then the equations we know that the phase must satisfy (\ref{systempsi}) along the path $\{(t,x(t)): 0\leq t\leq T\}$ become
\begin{align}\label{bich}
&\frac{dx(t)}{dt}=-h_{p}(x(t),t,\omega(t)) \qquad \frac{d\omega(t)}{dt}=h_x(x(t),t,\omega(t)) \qquad \frac{d\psi}{dt}(x(t),t)=0 
\\ \label{ricatti}
& \frac{dM}{dt}=A+BM+MB^t+MCM  
\end{align}
The last equation (\ref{ricatti}) is a matrix Riccati equation associated to (\ref{bich}). It is a non-linear equation which is not always well-posed. From the equation $\psi_t(x(t),t)=0$ in (\ref{bich}), and the initial condition, this implies $\psi(x(t),t)=x(t)\cdot \omega(t)$ and $\psi_1(t)=\omega(t)$ as claimed. The crucial choice is therefore the Hessian, $M(t)$ which is associated to the second order terms in $(x-x(t))$. We chose the initial condition $M(0)=iI$. We also associate the matrices $Y(t)$ and $N(t)$ to the Hessian $M(t)$. Now we let $Y(t)$ and $N(t)$ satisfy the following system: 
\begin{align}\label{systemY}
&\frac{dY}{dt}=-B^tY-CN \qquad \frac{dN}{dt}=AY+BN \\& 
(Y(0),N(0))=(I,iI) \nonumber
\end{align}
We claim whenever $(Y(t),N(t))$ is a solution to (\ref{systemY}), then $Y(t)$ is invertible, and the solution $M(t)=N(t)Y^{-1}(t)$ to (\ref{ricatti}) exists for all bounded time intervals, if and only if $M(t)$ is positive definite. With the given initial conditions (\ref{icephase}), this is equivalent to the claim we can find a phase function such that $\Im M(t)$ is positive definite. The compassion is now standard in the literature and we will not include it here, for details, see \cite{waterss}, Theorem 1, or \cite{Ralston}. 

Now we proceed to find the coefficients $a_j(x,t)$ of the beam. The linear operator $L$ is the transport operator which acts on functions $a(x,t)$ in the following manner
\begin{align*}
La=2\psi_t a_t-2g^{kl}\psi_{x_k}a_{x_l}+(\Box_g\psi)a.
\end{align*}
It is natural to consider $a_j(x,t)$ as a sum of homogeneous polynomial with respect to $x-x(t)$ as well, so we Taylor expand
\begin{align}\label{taylor}
a_j(x,t)=\sum\limits_{0\leq l}a_{j,l}(t)(x-x(t))^l.
\end{align} 
The natural number $l$, depends on the regularity of the coefficients and the regularity of the $g^{jk}$. For our purposes, we only use $l=0$ which suffices for this article. 
From this identity we can match up term in our Taylor series expansion. Using (\ref{bich}), we obtain a differential equation for $a_{j,l}(t)$;
\begin{align}\label{diffeq}
\frac{d}{dt}a_{j,l}(t)+r(t)a_{j,l}(t)=F_{j,l}(t). 
\end{align} 
The right hand side is a homogenous polynomial of order $l$ in $x-x(t)$ which depends on $a_{j,l}(t)$ and $\psi_{j}$. The factor $r(t)$ comes from computing $H\psi$ along the curves (\ref{bich}). We obtain ordinary differential equations defining $a_{j,l}(t)$ as follows
\begin{align*}
\frac{d}{dt}a_{j,l}(t)+\sparen{\frac{d}{dt}\sigma(t)}a_{j,l}(t)=F_{j,l}(t).
\end{align*}
for some $\sigma(t)$ such that $r(t)=\frac{d\sigma(t)}{dt}$. Solutions to these equations are given by 
\begin{align}\label{coeffeq}
a_{j,l}(t)=\sigma(t)\sparen{a_{j,l}(0)+\int\limits_0^t\sigma^{-1}(s)F_{j,l}(s)\,ds}
\end{align}
A lengthy computation in \cite{eskin}, Section 64.3 gives 
\begin{align}\label{sigma}
\sigma(t)=\frac{1}{\sqrt{|\frac{\partial x(t)}{\partial x_0}|}}\exp\left(\int\limits_0^tF(x(s),s)\,ds\right)
\end{align}
with 
\begin{align}
F(x,t)=-\frac{1}{\psi_t(x,t)}\sum\limits_{p,k=1}^ng_{x_k}^{pk}(x,t)\psi_{x_p}-\frac{1}{4\psi_t^2}\sum\limits_{p,k=1}^ng_t^{pk}(x,t)\psi_{x_p}\psi_{x_k}
\end{align}
When finding $\sigma(t)$, evaluating at $x=x(t)$, we note that on null bicharacteristics $\nabla\psi=\omega(t)$ and $\psi_t=\lambda_i(t,x(t))$, $i=1,2$. 

\begin{lem}\label{size}
We can write
\begin{align*}
a_{0}(t,x)=\sigma(t)a(0,x)+\mathcal{O}(|x-x(t)|).
\end{align*}
\end{lem}
with $\sigma(t)$ given by \eqref{sigma} above. 
\begin{proof}
We can compute the first few terms given by (\ref{coeffeq}). For the case of $C^2$ coefficients the error terms above are correct and make sense. Since we know that $F_{0,0}=0$, we compute
\begin{align}\label{a0}
a_{0,0}(t)=a_{0,0}(0)\sigma(t).
\end{align}
\end{proof}

Assuming that $H(x,t,\xi,\tau)=H(\infty,t,\xi,\tau)$ for $|x|>R$, then $w\equiv 0$ for all $|x|>R$. 
Assuming $\mathrm{diam}(\Omega')>R$, $w$ satisfies the hypothesis of Theorem \ref{energy}. Inserting $U_{\lambda}(x,t)$ into the wave equation we now consider the sets
\begin{align*}
& A_{\lambda}=\{x\in \Omega': |x-x(t)|\leq \sqrt{B}\lambda^{-1/2}\} \\&
   A_{\lambda}^c=\{x\in \Omega': |x-x(t)|>\sqrt{B}\lambda^{-1/2}\}
\end{align*}
where $B$ is a fixed positive constant.  

From Taylor's theorem, assuming $g^{jk}(x,t)$ has two derivatives, the phase can be expanded to second order around $|x-x(t)|$ with error which is $o|x-x(t)|^2$. The usual error is $\mathcal{O}|x-x(t)|^3$ for a two term expansion of a $C^3$ phase. The higher the regularity, the better the error estimates in Taylor's theorem, and the more accurate the Ansatz construction. 

Applying these definitions we see that 
\begin{lem}
In the set $A_{\lambda}$, each $t_{j}$ $j=1,2$ can be bounded as
\begin{align}
t_{j}=\mathcal{O}(\epsilon\lambda^2|x-x(t)|^2)+\mathcal{O}(\epsilon\lambda |x-x(t)|)
\end{align}
and in $A_{\lambda}^c$ 
\begin{align}
t_{j}=\mathcal{O}(\lambda^2|x-x(t)|^2)+\mathcal{O}(\lambda |x-x(t)|).
\end{align}
\end{lem}
\begin{proof} 
The first order terms depend on the $C^2$ norm of $g^{jk}$ and the $L^2(\Omega\times(0,T)))$ norm of $a(x,t)$, from solving \eqref{elk} up to $o(|x-x(t)|)$, using the Peano form of the remainder in Taylor's theorem. The second order terms depend on the $C^2$ norm of $g^{jk}$ and the $H^3(\Omega\times(0,T))$ norm of $a(x,t)$, from solving \eqref{trans} up to $o|x-x(t)|$. Here for the first terms, we have used the Peano form of the remainder, from Taylor's Theorem. When we reformulate the remainder this states that that $\forall \epsilon>0$ there exists a $\lambda(\epsilon)$ sufficiently large so that for the remainder, say $r_2(x)$, we have that
\begin{align}
|r_2(x)|\leq \epsilon|x-x(t)|^2
\end{align}
whenever $|x-x(t)|\leq \sqrt{B}\lambda^{-1/2}$ provided $\lambda\geq \lambda(\epsilon)$. Away from this region, in $A_{\lambda}^c$ it is possible to approximate the remainder, but only to $\mathcal{O}|x-x(t)|^2$. This is computation is the similar for the transport equation, whence the Lemma is proved. 
\end{proof}

\begin{proof}[Proof of Theorem \ref{ansatz2}]
We make the definitions: 
\begin{align}
\int\limits_b^{\infty}\exp(-x^2)\,dx=\mathrm{erfc}(b)  \qquad \int\limits_0^{b}\exp(-x^2)\,dx=\mathrm{erf}(b)
\end{align}
We know that the exponential function admits the following asymptotics:
\begin{align}\label{reallylarge}
\mathrm{erfc}(b)= \frac{\exp(-b^2)}{2b} +\mathcal{O}\sparen{\frac{\exp(-b^2)}{b^3}}
\end{align}
from Example 4 on page 255 of \cite{bender}, whenever $b$ is sufficiently large.  

We now proceed to estimate
\begin{align}
\lambda^{n/2}t_{j}\exp(i\lambda\psi_{j})
\end{align}
in $L^2(\Omega')$ norm. We know
\begin{align}\label{oE}
&\lambda^{n/2}\int\limits_{A_{\lambda}}\epsilon |x-x(t)|^2\lambda^2\exp(-\lambda |x-x(t)|^2)+\epsilon \lambda|x-x(t)|\exp(-\lambda |x-x(t)|^2)\,dx \\& \nonumber \leq C\lambda\left(B\epsilon+\frac{\sqrt{B}\epsilon}{\sqrt{\lambda}}\right)\mathrm{erf}(\sqrt{B})
\end{align}
with $C$ independent of $\lambda$ and $\epsilon$. 

We also have that 
\begin{align}\label{sE}
&\lambda^{n/2}\int\limits_{A_{\lambda}^c}|x-x(t)|^2\lambda^2\exp(-\lambda|x-x(t)|^2)+\lambda|x-x(t)|\exp(-\lambda|x-x(t)|^2)\,dx\leq \\& \nonumber C\lambda\mathrm{erfc}(\sqrt{B})
\end{align}
with $C$ independent of $\lambda$ and $\epsilon$ depending on diam$(\Omega')$ and $t$.

Combining, \eqref{sE} and \eqref{oE}, after making the change of variables $\sqrt{M(t)}|x-x(t)|\rightarrow |x-x(t)|$,  one sees that for $B,\lambda>1$, 
\begin{align}
||\lambda^{n/4}t_{j}\exp(i\lambda\psi_j)||_{L^2(\Omega')}\leq C\lambda\left(B\epsilon+\mathrm{erfc}(\sqrt{B})\right)
\end{align}
with $C$ independent of $\lambda$ and $\epsilon$ depending on diam$(\Omega')$ and $t$.

Moreover one has that using the energy estimates in Theorem \ref{energy} and the asymptotic \eqref{reallylarge}
\begin{align}
||w(x,t)||_{\dot{H}_0^1(\Omega'\times(0,T))}\leq \int\limits_0^T\sum\limits_{j=1}^2
||\lambda^{n/4}t_{j}\exp(i\psi_j)||_{L^2(\Omega')}\,dt\leq C_{22}\lambda\left(B\epsilon+\exp(-B)\right)
\end{align}
from which it follows from the form of the remainder:
\begin{align}
||w(x,t)||_{L^2(\Omega'\times[0,T]))}\leq C_{22}\left(B\epsilon+\exp(-B)\right)
\end{align}
The constant $B$ is arbitrary, so we can take the minimum in $B$ of the right hand side ($\sim \log (\epsilon^{-1}))$ to reach the desired conclusion. 
\end{proof}

Instead of the Gaussian beam tail, we no longer have the ability to construct the Hessian matrix. We use as our initial data as before
\begin{align}
(x-x_0)\cdot\omega_0+i|x-x_0|^2
\end{align}
but we propagate it as
\begin{align}\label{special}
(x-x(t))\cdot\omega(t)+i|\omega(t)||x-x(t)|^2)
\end{align}

In order to prove Theorem \ref{ansatz1} we need to be able to solve the eikonal only to first order. We apply the operator $H$ to the function and we obtain the following error terms.
\begin{lem}
We have that 
\[
|H_0(t,x,\psi_x,\psi_t)|\leq 10M_0|\omega(t)|^2|x-x(t)|^2
\]
\end{lem}
\begin{proof}
This is a refinement of Lemma 3 in \cite{watersp}, . As $H_0(t,x,\psi_x,\psi_t)$ is positive homogeneous of degree two in $|\omega(t)|$, we start by showing on nul-bicharacteristics $(t,x(t))$ 
\[
\nabla_{x}H_0(t,x(t),\psi_x(t,x(t)),\psi_t(t,x(t)))=0.
\]
and then refine the $o(|x-x(t)|)$ terms from the Peano form of the remainder. Computing $\nabla_{x}H_0(t,x,\psi_x,\psi_t)$,
\begin{equation}\label{expression}
\frac{\partial}{\partial x_j}H_0(t,x,\psi_x,\psi_t)=H_{0x_j}+H_{0\xi_l}\psi_{x_lx_j}+H_{0\tau}\psi_{tx_j}.
\end{equation}
Dividing (\ref{expression}) by $2\lambda_1$ or $(2\lambda_2)$ and substituting the equations in (\ref{bich}) into the right hand side of (\ref{expression}) we obtain for the right hand side of (\ref{expression})
\begin{equation}\label{input}
-\frac{d\omega_j}{dt}+ \frac{dx_l}{dt}\psi_{x_lx_j}+\psi_{tx_j}
\end{equation}
As $\psi_{x_j}(t,x(t))=\omega_j(t)$, differentiating $\omega_j(t)$ with respect to $t$  we have
\begin{equation}\label{dif}
\frac{d\omega_j}{dt}= \frac{dx_l}{dt}\psi_{x_lx_j}+\psi_{tx_j}.
\end{equation}
Substituting (\ref{dif}) into (\ref{input}) implies (\ref{input}) is 0, which happens if and only if (\ref{expression}) vanishes on nul-bicharacteristics. 

We can now write
\begin{align}
&H_0(x,t,\psi_x(t,x),\psi_t(t,x))=\\& \nonumber H_0(x(t),t,\psi_x(t,x(t)),\psi_t(t,x(t)))+\nabla_xH_0(x(t),t,\psi_x(t,x(t)),\psi_t(t,x(t)))+\\& \nonumber r_1(x)|x-x(t)|
\end{align}
We abbreviate $H_0(x(t),t,\psi_x,\psi_t)=H_0(x(t))$ We note then that if $h=x-x(t)$, the remainder $r_1(x)$ is then
\begin{align}\label{remainder}
r_1(x)= \frac{H_0(x(t)+h)-H_0(x(t))}{h}-\nabla_xH_0(x(t))
\end{align}
The result will follow if we can prove 
\begin{align}
|r_1(x)|\leq 10M_0|\omega(t)|^2h
\end{align}
because we just showed $\nabla_xH_0(x(t))=0$ and $H_0(x(t))=0$ by definition of the ODE's \eqref{bich}. 
The mean value theorem applied to \eqref{remainder} implies that it suffices to prove for some 
$c\in (0,h)$
\begin{align}
|\nabla_xH_0(x(t)+c)-\nabla_xH_0(x(t))|\leq 10M_0h|\omega(t)|^2
\end{align}
But this is follows as setting $x(t)+c=\tilde{c}$, we have that 
\begin{align}
\left|\nabla_x\left(\sum\limits_{jk}g^{jk}\psi_{x_j}\psi_{x_k}\right)(\tilde{c})-\nabla_x\left(\sum\limits_{jk}g^{jk}\psi_{x_j}\psi_{x_k}\right)(x(t))\right|\leq 2M_0|\omega(t)|^2h
\end{align}
where we used the fact we can expand $\psi_{x_j}\psi_{x_k}$ one term in $x$, and also  
\begin{align}
\left|\nabla_x(\psi_t^2)(\tilde{c})-\nabla_x(\psi_t^2)(x(t))\right|\leq 2M_0|\omega(t)|^2h
\end{align}
which follows by a short computation writing down the explicit form of $\psi_t$, and using \eqref{bich}, with $\psi_t(t,x(t))=\lambda_1(t,x(t))$. 
\end{proof}

\begin{proof}[Proof of Theorem \ref{ansatz1}]
We then claim that $U_{\lambda}$ is a parametrix solution to \eqref{cauchy}, but the phase functions now have the special form above in \eqref{special}. In particular, using the Peano form of the remainder, but this time with only 1 term in the Taylor expansion, the error term from solving \eqref{elk} is then 
\begin{align}
\lambda^{n/2}\int\limits_{A_{\lambda}}M_0\lambda^2|x-x(t)|^2\exp(-\lambda|x-x(t)|^2)\,dx\leq C\lambda BM_0
\end{align}
and 
\begin{align}
\lambda^{n/2}\int\limits_{A_{\lambda}^c}\lambda^2|x-x(t)|\exp(-\lambda|x-x(t)|^2)\,dx\leq C\lambda \mathrm{erfc}(-\sqrt{B})
\end{align}
and the other terms from solving the transport equation \eqref{trans} are bounded in a similar way. From Theorem \ref{energy} and Corollary \ref{size}, we can again conclude the desired result after using the asymptotic for the error function \eqref{reallylarge} and minimising in $B$. 
\end{proof}

\section{Proof of Observability Estimates}
We can write $||u||_{L^2(\Omega_0\times[0,T])}^2$ as 
\begin{align}
\lambda^{n/2}\int\limits_0^T\int\limits_{\Omega_0}\left(\sum\limits_{j=1}^2(a_{0jl}\exp(i\lambda\psi_{jl})\right)^2\,dx\,dt+E_l
\end{align}
where $l$ denotes the regularity of the coefficients, with $l=1$ corresponding to $C^{1,1}$ coefficients and $l=2$ corresponding to $C^2$ coefficients. The term $E_l$, $l=1,2$ is bounded by Theorems \ref{ansatz1} and \ref{ansatz2} respectively. The cross terms involving $\exp(i\lambda\psi_1)\overline{\exp(i\lambda\psi_2)}$ are oscillatory and can be discarded by stationary phase. Indeed, separating the phase into real and imaginary parts, applying Lemma \ref{stationary} and using the integrals \eqref{GB}, gives 
\begin{align}
|\lambda^{n/2}\int\limits a_{01l}\overline{a_{02l}}\exp(i\lambda\psi_1)\overline{\exp(i\lambda\psi_2)}\,dx|\leq \frac{C}{\sqrt{\lambda}}||a_{01l}a_{02l}||_{C^1(\Omega)}
\end{align}
with $C$ depending on $\Omega$, and $||g^{jk}||_{C^{1,1}}$. 
We use the fact that 
\begin{align}
(b-d)^2+(b+d)^2=2(b^2+d^2)
\end{align}
along with the initial conditions to conclude 
\begin{align}\label{estimate1}
&\lambda^{n/2}\int\limits_{0}^T\int\limits_{\Omega_0}\sigma_l^2(t)\left(f_1^2(x(0))+\frac{f_2^2(x(0))}{\lambda_1^2(0,x(0))}\right)\exp(-\lambda\beta_l(t)|x-x(t)|^2)\,dx\,dt\leq \\& ||u||_{L^2(\Omega_0\times[0,T])}^2+E_l \nonumber
\end{align}
where $\beta_2(t)=2\Im M(t)$ and $\beta_1(t)=|\omega(t)|$.  
Now we need the following Lemmas:
\begin{lem}\label{varchange}
We have that $\sigma_l(t)(\beta_l(t))^{n/2}=C_l(t)>0$ with $l=1,2$ and $\sigma_l(t)$ given by \eqref{sigma}. 
\end{lem}
\begin{proof}
We need to solve $H\psi=0$ to first order around $x(t)$, as indicated earlier. The inequalities are direct as $\sigma_l(t)$, is positive for both $l=1,2$ and $|\omega(t)|$ is positive and $\Im M(t)$ is positive definite. We also remark that the inequality for $l=2$ can also be found as a result of \cite{LKK}, Lemma 2.58.
\end{proof}

\begin{proof}[Proof of Theorem \ref{low1Rn2}]
Whenever $x\in A_{\lambda}^c$ the main term is small, as we know that
\begin{align}\label{far}
\lambda^{n/2}\int\limits_{A_{\lambda}^c\cap \Omega_0}\exp(-\lambda|x-x(t)|^2)\,dx\,dt\leq \mathrm{erfc}(\sqrt{B})
\end{align} 
if $B=B_{min}(\epsilon)$. 

If $x\in \Omega_0\cap A_{\lambda}$ then necessarily $x(t)\in \Omega_0$ as $\lambda\rightarrow\infty$. We know that $x(0)\mapsto x(t)$ is connected, as $x(t)$ traces out a continuous curve. Thus the GCC is automatically satisfied to get any meaningful estimate. One necessarily has 
\begin{align}
C_2(T)f_1^2(x(0))\mathrm{erf}(\sqrt{B})\leq ||u||_{L^2(\Omega_0\times [0,T])}^2+E_{2}
\end{align}
and similarly for $f_2^2(x(0))/\lambda_1^2(0,x(0))$, with $C_2(T)=\int\limits_0^TC_2(t)\,dt$.

We now obtain
\begin{align}
C_2(T)\left(f_1^2(x(0))+\frac{f_2^2(x(0))}{\lambda_1^2(0,x(0))}\right)\left(\mathrm{erf}{\sqrt{B}}-\mathrm{erfc}{\sqrt{B}}\right)\leq ||u||_{L^2(\Omega_0\times [0,T])}^2+C_{21}C_2(\epsilon)
\end{align}
When $B=B_{\min}(\epsilon)$, then the factor $\left(\mathrm{erf}{\sqrt{B}}-\mathrm{erfc}{\sqrt{B}}\right)=1-2\mathrm{erfc}{\sqrt{B}}$ is very nearly 1, by the asymptotic \eqref{reallylarge}. Applying the convergence Lemma \ref{H} in the appendix to approximate $f_1^2(x(0))$ and $f_2^2(x(0))/\lambda_1^2(0,x(0))$ gives the desired result.  
\end{proof}

\begin{proof}[Proof of Theorem \ref{low1Rn2}]
The steps of the proof follow exactly the same using the appropriate inequality in Lemma \ref{varchange}, except for the fact that the error term $E_1$ is now bounded by $C_1(M_0)$ with $C_1(M_0)\rightarrow 0$ as $M_0\rightarrow 0$. Thus, the parameter $M_0$ takes the place of the parameter $\epsilon$ in the above proof. 
\end{proof}

\section{Appendix: Convergence Lemmas}


We prove the following, similar to \cite{waterss}: 
\begin{lem}\label{H}
Let $h(t,x)\in C^1((0,T)\times O)$, where $O$ is an open subset of $\mathbb{R}^n$ and $B$ be a symmetric nonsingular matrix such that $\Re{B}\geq 0$, if $x(t)$ is a continuous curve defined in terms of $t$ in $O$, then we have the following uniform estimate
\begin{align}\label{rrep}
&\sabs{\sparen{\frac{\lambda}{\pi}}^{\frac{n}{2}}(\det B)^{\frac{1}{2}}\int\limits_{O}\exp\sparen{\langle-\lambda B(x-x(t)), (x-x(t)\rangle} h(t,x)\,dx-h(t,x(t))}< \\& \nonumber
\sparen{\frac{2\lambda^{\sigma}}{\sqrt{\lambda}}+4\mathrm{erfc}(-\lambda^{2\sigma})}\norm{h(t,x)}_{C^1((0,T)\times O)}
\end{align}
with $\sigma\in (0,1/6)$.
\end{lem}
\begin{proof}
 The assumption that $h(t,x)$ is in $C^1([0,T]\times O)$ implies that $h(t,x)$ is locally uniformly Lipschitz continuous with Lipschitz constant $\norm{h(t,x)}_{C^1((0,T)\times O)}$. We set $\epsilon=\lambda^{\sigma-1/2}\norm{h(t,x)}_{C^1((0,T)\times O)}$.  We know that for $\eta=\lambda^{\sigma-1/2}$, if $x$ is such that $|x-x(t)|<\eta$, this implies
\begin{align*}
\sabs{h(t,x)-h(t,x(t))}<2\frac{\lambda^{\sigma}}{\sqrt{\lambda}}\norm{h(t,x)}_{C^1((0,T)\times O)}
\end{align*}
Using change of variables, we then obtain the bounds
\begin{align*}
&\sabs{\sparen{\frac{\lambda}{\pi}}^{\frac{n}{2}}(\det B)^{\frac{1}{2}}\int\limits_{O}\exp\sparen{\langle-\lambda B(x-x(t)), (x-x(t)\rangle} h(t,x)\,dx-h(t,x(t))}<\\& \nonumber \frac{2\lambda^{\sigma}}{\sqrt{\lambda}}\norm{h(t,x)}_{C^1((0,T)\times O)}\int\limits_{|y|\leq C\eta}\sparen{\frac{\lambda}{\pi}}^{\frac{n}{2}}\exp(-\lambda |y|^2)\,dy+\\& \nonumber
2\norm{h(t,x)}_{C^0((0,T)\times O)}\int\limits_{C\eta<|y|< \infty}\sparen{\frac{\lambda}{\pi}}^{\frac{n}{2}}\exp(-\lambda |y|^2)\,dy \leq \\& \nonumber
 \sparen{\frac{2\lambda^{\sigma}}{\sqrt{\lambda}}+ 4\mathrm{erfc}(-\lambda^{2\sigma})}\norm{h(t,x)}_{C^1((0,T)\times O)}
\end{align*}
Here we notice that normalization factor of $(\det B)^{1/2}$ makes the Gaussian kernel normalized to $1$.
\end{proof}

We recall the following elementary Lemmas
\begin{lem}
Suppose $\Omega\subset\mathbb{R}^n$ is open, $f:\Omega\rightarrow \mathbb{R}$ is $C^{\infty}$, $p\in \Omega$ and $\nabla f(p)\neq 0$. Then there are neighborhoods $U$ and $V$ of $0$ and $p$ respectively and a $C^{\infty}$ diffeomorphism $G: U\rightarrow V$ with $G(0)=p$ and 
\begin{align}
f\circ G(x)=f(p)+x_n
\end{align}
\end{lem}
We also see that 
\begin{lem}\label{stationary}
Let $\phi$ be a real valued $C^{\infty}$ function and let $v$ be a $C_0^{\infty}$ function and define 
\begin{align}
I(\lambda)=\int\exp(-\pi i\lambda\phi(x))a(x)v(x)\,dx
\end{align}
here $\lambda>0$ is a large scalar and $a(x)\in W^{N,1}(\Omega)$. Suppose $\Omega\subset \mathbb{R}^n$ is open, $\phi:\Omega\rightarrow \mathbb{R}$ is $C^{\infty}, p\in \Omega$ and $\nabla\phi(p)\neq 0$. Suppose that $v\in C_0^{\infty}$ has its support in a sufficiently small neighborhood of $p$ then 
\begin{align}
\forall N, \exists C_n: |I(\lambda)|\leq C_N\lambda^{-N}
\end{align}
and furthermore $C_N$ depends only on bounds for $N+1$ derivatives of $\phi$, the $W^{N,1}(\Omega)$ norm of $a(x)$ and a lower bound for $|\nabla\phi(p)|$. 
\end{lem}
\begin{proof}
Let $\phi_1=\phi_2\circ G$ where $G$ is a smooth diffeomorphism. Then we have
\begin{align*}
&\int\exp(-\pi i\lambda\phi_2(x))a(x)v(x)\,dx=\\& \int\exp(-\pi i\lambda\phi_1(G^{-1}x))a(x)v(x)\,dx=\\&
\int\exp(-\pi i\lambda\phi_1(y))a(Gy)v(Gy)d(Gy)=\\&
\int\exp(-\pi i\lambda\phi_1(y))a(Gy)v(Gy)|J_G(y)|\,dy
\end{align*}
where $J_G$ is the Jacobian determinant. 
The straightening lemma and the calculation reduce this to the case where $\phi(x)=x_n+c$. In this case, letting $e_n=(0, . . .,0,1)$ we have
\begin{align}
I(\lambda)=\exp(-i\pi\lambda c)\hat{av}(\frac{\lambda}{2}e_n) 
\end{align}
and this has the requiste decay as 
\begin{align}
\widehat{D^{\alpha}f}(\xi)=|\xi|^{\alpha}\hat{f}(\xi)
\end{align}
\end{proof}

We recall the following set of 1-dimensional Gaussian integrals
\begin{align}\label{GB}
&\int\limits_0^{\infty}x^{2n}\exp(-ax^2)\,dx
=\sqrt{\frac{\pi}{a}}\frac{(2n-1)!!}{a^n2^{n+1}}\quad 
\int\limits_0^{\infty}x^{2n+1}\exp(-ax^2)\,dx=\frac{n!}{2a^{n+1}}\\&
\int\limits_{-\infty}^{\infty}x^{2n}\exp(-ax^2)\,dx=\sqrt{\frac{\pi}{a}}\frac{(2n-1)!!}{(2a)^n} \nonumber
\end{align}

\section*{Acknowledgments}
A.~W.~acknowledges support by EPSRC grant EP/L01937X/1.

\end{document}